\documentclass[12pt]{amsart}
\usepackage[english]{babel} 
\usepackage{url}
\usepackage{hyperref}
\usepackage[mathscr]{eucal}
\usepackage{mathrsfs}
\usepackage{amssymb,amsmath,amsthm}
\newtheorem{theorem}{Theorem}[section]
\newtheorem{lemma}[theorem]{Lemma}
\newtheorem{corollary}[theorem]{Corollary}
\newtheorem{proposition}[theorem]{Proposition}

\newtheorem{remark}[theorem]{Remark}

\def\N{\mathbb{N}}
\def\Q{\mathbb{Q}}
\def\Z{\mathbb{Z}}
\def\F{\mathbb{F}}

\let\ds=\displaystyle
\let\ov=\overline

\let\wt=\widetilde

\def\Cl{{\mathcal C}\hskip-2pt{\ell}}
\newcommand{\cl}{c\hskip-1pt{\ell}}
\def\order{\raise1.5pt \hbox{${\scriptscriptstyle \#}$}}
\def\lien{\mathrel{\mkern-4mu}}
\def\too{\relbar\lien\rightarrow}
\def\tooo{\relbar\lien\relbar\lien\too}

\newcommand{\plus}{\ds\mathop{\raise 2.0pt \hbox{$\bigoplus $}}\limits}
\newcommand{\prd}{\ds\mathop{\raise 2.0pt \hbox{$  \prod   $}}\limits}
\newcommand{\sm}{\ds\mathop{\raise 2.0pt \hbox{$  \sum    $}}\limits}
\def\Sauf{\setminus}

\def\Kappa{\hbox{\large$\kappa$}}
\def\Nu{\hbox{\large$\nu$}}

\begin{document}

\markboth{Georges Gras}{}

\title[The $p$-rank $\varepsilon$-conjecture for towers of $p$-extensions]
{The $p$-rank $\varepsilon$-conjecture on class groups \\ is true 
for towers of $p$-extensions}

\author{Georges Gras}
\address{Villa la Gardette, 4 Chemin Ch\^ateau Gagni\`ere,
F-38520 Le Bourg d'Oisans}
\email{g.mn.gras@wanadoo.fr}

\begin{abstract} Let $p \geq 2$ be a given prime number.
We prove, for any number field $\Kappa$ and any integer 
$e \geq 1$, the $p$-rank $\varepsilon$-conjecture, on the $p$-class 
groups $\Cl_F$, for the family ${\mathscr F}_{\!\!\kappa}^{p^e}$
of towers $F/\Kappa$ built as successive degree $p$ cyclic extensions 
(without any other Galois conditions) such that $F/\Kappa$ be of 
degree $p^e$, namely: $\order (\Cl_F \otimes \F_p) 
\ll_{\kappa,p^e,\varepsilon} (\sqrt{D_F}\,)^\varepsilon$ 
for all $F \in {\mathscr F}_{\!\!\kappa}^{p^e}$, where $D_F$ 
is the absolute value of the discriminant (Theorem \ref{thmf}),
and more generally $\order (\Cl_F \otimes \Z/p^r\Z)
\ll_{\kappa,p^e,\varepsilon} (\sqrt{D_F}\,)^\varepsilon$ ($r \geq 1$ fixed).
This Note generalizes the case of the family ${\mathscr F}_{\!\Q}^{p}$ ({\it Genus theory 
and $\varepsilon$-conjectures on $p$-class groups}, J. Number Theory 207, 
423--459 (2020)), whose techniques appear to be ``universal'' for all relative degree
$p$ cyclic extensions and use the Montgomery--Vaughan result on prime numbers. 
Then we prove, for ${\mathscr F}_{\!\!\kappa}^{p^e}$, the $p$-rank 
$\varepsilon$-conjecture on the cohomology groups ${\rm H}^2({\mathcal G}_F,\Z_p)$
of Galois $p$-ramification theory over $F$ (Theorem \ref{thmf2}) and for some
other classical finite $p$-invariants of $F$, as the Hilbert kernels and the logarithmic 
class groups.
\end{abstract}


\date{February 20, 2020} 

\maketitle

\section{Introduction}

For any prime number $p \geq 2$ and any number field $F$, we denote by $\Cl_F$ the 
$p$-class group of $F$ (in the restricted sense for $p=2$); the precise sense
does not matter, the case of ordinary sense being a consequence. 

\smallskip
To avoid any ambiguity, 
we shall write $\Cl_F \otimes \F_p$, isomorphic to the ``$p$-torsion group''
also denoted $\Cl_F[p]$ in the literature, only giving the $p$-rank  of $\Cl_F$:
$${\rm rk}_p(\Cl_F) := {\rm dim}_{\F_p}(\Cl_F/\Cl_F^p). $$

We refer to our paper \cite{Gr00} for an introduction with some history about the notion 
of $\varepsilon$-conjecture, initiated by Ellenberg--Venkatesh \cite{EV} 
by means of reflection theorems \cite{Gr5} and Arakelov class groups \cite{Sc},
then developed by many authors as Frei--Pierce--Turnage-Butterbaugh--Widmer--Wood
\cite{EPW,EV,FW,PTW1,PTW2,W,Wid},\,\ldots, then 
the related density results of Koymans--Pagano \cite{KP}, all these questions 
being in relation with the classical heuristics/conjectures on class groups 
\cite{AM,CL,CM,DJ,Ge,Ma}.

\medskip
We prove, unconditionally:

\medskip \noindent
{\bf Main results.} {\it Let $\Cl_F$ denotes the $p$-class group 
of a number field $F$ and $D_F$ the absolute value of its discriminant. 
For $\Kappa$ and $e$ fixed, the general $p$-rank $\varepsilon$-conjecture 
is true for the family ${\mathscr F}_{\!\!\kappa}^{p^e}$ of  
$p$-towers $F/\Kappa$ of degree $p^e$
($\Kappa =F_0 \subset \cdots F_{i-1} \subset F_i \cdots \subset F_e=F$) 
with $F_i/F_{i-1}$ cyclic of degree $p$ for all $i \in [1,e]$ 
(the $F/\Kappa$ will be called ``$p$-cyclic-towers''): 
for all $\varepsilon > 0$, there exists a constant 
$C_{\kappa ,p^e,\varepsilon}$, such that  (Theorem \ref{thmf}):
$$\order (\Cl_F \otimes \F_p) 
\leq C_{\kappa ,p^e,\varepsilon} \cdot (\sqrt{D_F}\,)^\varepsilon , 
\ \ \hbox{for all $F \in {\mathscr F}_{\!\!\kappa}^{p^e}$.}$$ 

Moreover, the parameter $\Kappa$ only intervenes by its degree 
$d_\kappa \leq d$ and the $p$-rank ${\rm rk}_p(\Cl_\kappa) \leq \rho$ 
of its class group (for $d$ and $\rho$ given), so that, in some sense,
$C_{\kappa ,p^e,\varepsilon} = C_{d,\rho ,p^e,\varepsilon}$.

\smallskip\noindent
We obtain the same result for the cohomology groups 
${\rm H}^2({\mathcal G}_F,\Z_p)$ of $p$-ramification theory, where
${\mathcal G}_F$ is the Galois group of the maximal $p$-ramified 
pro-$p$-extension of $F$, then, for example, for $p$-adic regulators, Hilbert's and 
regular kernels, Jaulent's logarithmic class groups 
(Theorem \ref{thmf2} and Remark \ref{divers}).}

\medskip\noindent
{\bf Note.} During the writing of this article, we have been informed of papers by 
Wang \cite[Theorem 1.1]{W} (dealing with the non-trivial case $\ell \ne p$ and mentioning, 
without proofs, the easier case $\ell = p$, for ${\rm Gal}(F/\Q) \simeq (\Z/p\Z)^e$) 
and Kl\"uners--Wang \cite{KW} giving a proof for $\ell = p$ and extensions $F/\Kappa$ 
contained in (Galois) $p$-extensions; their results are based on other 
information (Cornell paper \cite{Cor} using genus theory in elementary 
$p$-extensions, then a generalization of Cornell approach in \cite{KW}). 
We thank Jiuya Wang for these communications, in particular the text \cite{KW} 
(in order to be published) from a lecture by the~authors.

\smallskip\noindent
Our result, using another point of view, is unconditional with computable 
constants and without any Galois conditions, contrary to \cite{KW} where 
the ``$\varepsilon$-inequality'' is valid for $D_F \gg 0$ and where the 
Galois closure of $F/\Kappa$ is assumed to be a $p$-extension.
In fact, if our result does not give the ``strong $\varepsilon$-conjecture''
$$\order (\Cl_F \otimes \Z_p) \ll_{\kappa ,p^e,\varepsilon} 
\cdot (\sqrt{D_F}\,)^\varepsilon, $$ 

it gives, for any fixed $r\geq 1$, the $\varepsilon$-inequality
(see Corollary \ref{CRank}):
$$\order (\Cl_F \otimes \Z/p^r\Z) \ll_{\kappa ,p^e,\varepsilon} 
\cdot (\sqrt{D_F}\,)^\varepsilon, \ \ 
\hbox{for all $F \in {\mathscr F}_{\!\!\kappa}^{p^e}$.}$$ 

However, this suggests well, considering also the density results of \cite{KP}, 
that the strong $\varepsilon$-conjecture is true ``for almost all'' elements of 
${\mathscr F}_{\!\!\kappa}^{p^e}$.

\section{Principle of the method}
We will perform an induction on the successive degree $p$ cyclic 
extensions of a tower $F \in {\mathscr F}_{\!\!\kappa}^{p^e}$,
the principles for such $p$-cyclic steps coming from \cite{Gr00}
dealing with the family ${\mathscr F}_{\!\!\Q}^{p}$. 
The method involves using fixed points exact sequences, for the 
invariants considered, and the definition of ``minimal relative discriminants'' 
built by means of the Montgomery--Vaughan result on prime~numbers.

\subsection{Tower of degree $p$ cyclic fields and relative class groups}
Let $\Kappa$ be any fixed number field and let:
$$\hbox{$F_0 = \Kappa \subset F_1 \cdots  F_{i-1} \subset F_i \cdots
\subset F_e=F$, $e \geq 1$,}$$
be a $p$-cyclic-tower of fields with ${\rm Gal}(F_i/F_{i-1}) \simeq \Z/p\Z$, 
for $1 \leq i \leq e$ (by abuse, we use the word of tower even if the fields $F_i$ 
are not necessarily Galois over $\Kappa$). 

\begin{remark}
{\rm For $p=2$, all the $2$-cyclic-towers are obtained inductively by means of
Kummer extensions
$F_i = F_{i-1}(\sqrt{a_{i-1}})$, $a_{i-1} \in F_{i-1}^\times \Sauf F_{i-1}^{\times 2}$;
for $p \ne 2$, one uses the classical Kummer process, considering
$F'_{i-1}:= F_{i-1}(\mu_p)$ and then taking $a'_{i-1} = b'^{\,e_\omega}_{i-1}$
in $F'_{i-1}$, where $e_\omega$ is the idempotent of $\Z_p[{\rm Gal}(F'_{i-1} / F_{i-1})]$ 
corresponding to the Teichm\"uller character $\omega$, so that $F'_i:=F'_{i-1}(\sqrt{a'_{i-1}})$
is decomposed over $F_{i-1}$ into $F_i/ F_{i-1}$ cyclic of degree $p$. 

\smallskip
We check easily that, for any $p$, the (non-$p$-cyclic) towers of degree $p^e$, of the form 
$F_i = F_{i-1}(\sqrt[p]{a_{i-1}})$, fulfill the $p$-rank $\varepsilon$-conjecture.

\smallskip
This gives many more extensions $F/\Kappa$ of degree $p^e$ which are not 
necessarily (Galois) $p$-extensions.}
\end{remark}

Put $k=F_{i-1}$, $K=F_i$ and $G={\rm Gal}(K/k) =: \langle \sigma \rangle$.
We shall use the obvious general exact sequence:
$$1 \to \Cl_K^* \too \Cl_K \mathop{\too}^{\nu} \Cl_K^{\, \nu}, $$
where $\Nu := 1+\sigma + \cdots + \sigma^{p-1}$ is the algebraic norm and
$\Cl_K^* = {\rm Ker}(\Nu)$. 

\smallskip
Recall that $\Nu = {\rm J} \circ {\rm N}$, where ${\rm N}$ is the arithmetic 
norm (or the restriction  of automorphisms
${\rm Gal}(H_K/K) \to {\rm Gal}(H_k/k)$ in the 
corresponding Hilbert's class fields $H_K$, $H_k$), 
which yields ${\rm N}(\Cl_K) \subseteq \Cl_k$, and where 
${\rm J} : \Cl_k \to \Cl_K$ 
comes from extension of ideals (or is the 
transfer map ${\rm Gal}(H_k/k) \to {\rm Gal}(H_K/K)$); so 
$\Cl_K^{\,\nu} \simeq \Cl_{k}$ if and only if ${\rm N}$ is 
surjective (equivalent to $H_k \cap K = k$) and ${\rm J}$ 
injective (no capitulation).
Whence the following inequality for the $p$-ranks:
\begin{equation}\label{inequality}
{\rm rk}_p(\Cl_K) \leq {\rm rk}_p(\Cl_{k}) + {\rm rk}_p(\Cl_K^*).
\end{equation}

So, the main problem is to give an explicit upper bound of
${\rm rk}_p(\Cl_K^*)$ only depending on ${\rm rk}_p(\Cl_{k})$ and 
the number of ramified prime numbers in $K/k$. For this, we shall recall some
elementary properties of the finite $\Z_p[G]$-modules annihilated by 
$\Nu$ \cite{Gr2}.

\smallskip
Let $G = \langle \sigma \rangle \simeq \Z/p\Z$.
We consider a $\Z_p[G]$-module $M^*$ (of finite type) annihilated by 
$\Nu=1+ \cdots + \sigma^{p-1}$ (which will be $\Cl_K^*$)
for which we define the following filtration, for all $h \geq 0$ and $M^*_0=1$:
\begin{equation}\label{filtre}
M^*_{h+1}/M^*_h := (M^*/M^*_h)^G. 
\end{equation}

For all $h \geq 0$, $M^*_h = \{x \in M^*,\  x^{(1-\sigma)^h}=1 \}$\,\footnote{Note 
that if $M$ is a $\Z_p[G]$-module of finite type and $M^* = {\rm Ker}(\nu)$,
we have, in an obvious meaning, $(M^*)_h = (M_h)^*$, 
whence the writing $M^*_h$.},
the $p$-groups $M^*_{h+1}/M^*_h$ are elementary and the maps
$\ds M^*_{h+1}/M^*_h \,\mathop{-\!\!\!\too\,}^{1-\sigma} M^*_h/M^*_{h-1}$
are injective, giving a decreasing sequence for the orders 
$\order (M^*_{h+1}/M^*_h)$ as $h$ grows; whence:
\begin{equation}\label{ineqfond}
\order (M^*_{h+1}/M^*_h) \leq \order M^*_1, \ \, \hbox{for all $h \geq 0$}. 
\end{equation}

Since $M^*$ is a $\Z_p[G]/(\Nu)$-module and since
$\Z_p[G]/(\Nu) \simeq \Z_p[\zeta]$,
where $\zeta$ is a primitive $p$th root of unity, we may
write for ${\mathfrak p} = (1-\zeta)\,\Z_p[\zeta]$ and $s \geq 0$:
$$M^* \simeq \plus_{j=1}^s \Z_p[\zeta]/ {\mathfrak p}^{n_j},\ \, 
n_1 \leq n_2 \leq \cdots \leq n_s, $$
and the sub-modules $M^*_h$ are, in this $\Z_p[\zeta]$-structure, the following ones:
$$M^*_h \simeq \plus_{j,\ n_j \leq h} \Z_p[\zeta]/ {\mathfrak p}^{n_j},\ \,\hbox{for all $h \geq 0$}.$$

\begin{proposition} \label{propo1}
Under the condition $M^*{}^{\nu}=1$, the $p$-rank  of $M^*$ fulfills the inequality 
${\rm rk}_p(M^*) \leq (p-1) \cdot {\rm rk}_p (M^*_1)$.
\end{proposition}

\proof
Let $M^*[p] := \{x \in M^*, \ x^p=1\}$; since 
$p\,\Z_p[\zeta] = {\mathfrak p}^{p-1}$, we obtain that $M^*[p] = M^*_{p-1}$; 
then (from \eqref{ineqfond}):
$$\order M^*_{p-1} = \prd_{h=0}^{p-2} \order(M^*_{h+1}/M^*_h) \leq (\order M^*_1)^{p-1}.$$
Since $M^*_{p-1}$ and $M^*_1$ are elementary, the inequality 
on the $p$-ranks follows.
\qed \medskip

Considering $M=\Cl_K$ and $M^*$ in $K/k$ yields:
\begin{corollary} \label{coro0}
Let $K/k$ be a degree $p$ cyclic extension of number fields and let
$G={\rm Gal}(K/k)$. Then ${\rm rk}_p(\Cl_K^*) \leq 
(p-1)\,{\rm rk}_p(\Cl_K^{G*}) = (p-1)\,{\rm rk}_p(\Cl_K^G)$.
\end{corollary}

\subsection{Majoration of ${\rm rk}_p(\Cl_K^G)$ and ${\rm rk}_p(\Cl_K)$}

\begin{proposition} \label{propo2}
We have ${\rm rk}_p(\Cl_K^G) \leq {\rm rk}_p(\Cl_k) + t_k + {\rm rk}_p ( E_k/E_k^p)$, 
where $t_k$ is the number of prime ideals of $k$ ramified in $K/k$ and
where $E_k$ is the group of units of $k$.
\end{proposition}

\proof 
We have the classical exact sequence:
\begin{equation}\label{clg}
1 \to \cl_K(I_K^G) \too \Cl_K^G  \mathop{\too}^{\theta}
E_k \cap {\rm N}_{K/k} (K^\times)/ {\rm N}_{K/k}(E_K) \to 1, 
\end{equation}
where $I_K$ is the $\Z[G]$-module of ideals of $K$ and
where $\theta$ associates with $\cl_K({\mathfrak A})$, 
such that ${\mathfrak A}^{1-\sigma}=(\alpha)$, $\alpha \in K^\times$, the 
class of the unit ${\rm N}_{K/k}(\alpha)$ of $k$, modulo ${\rm N}_{K/k}(E_K)$. 
The surjectivity and the kernel are immediate.

\smallskip
The $\Z[G]$-module $I_K^G$ is generated by ${\rm J}(I_k)$,
extension in $K$ of the ideals of $k$, and 
by the $t_k$ prime ideals of $K$ ramified in $K/k$.
Thus, using the obvious inequality ${\rm rk}_p 
(E_k \cap {\rm N}_{K/k} (K^\times)/ {\rm N}_{K/k}(E_K)) 
\leq {\rm rk}_p (E_k/E_k^p)$, we obtain:

\smallskip
\centerline{${\rm rk}_p (\Cl_K^G) \leq {\rm rk}_p (\cl_K(I_K^G)) +
{\rm rk}_p (E_k \cap {\rm N}_{K/k} (K^\times)/ {\rm N}_{K/k}(E_K))$}

\medskip
\centerline{\hspace{-1.8cm}$\leq {\rm rk}_p (\Cl_k) + t_k + {\rm rk}_p (E_k /E_k^p)$,}

\smallskip\noindent
which proves the claim.
\qed \medskip

\begin{remark} {\rm The Chevalley formula \cite{Ch} gives, in $K/k$, the order:
$$\order( \Cl_K^G) = \order \Cl_k \cdot 
\frac{p^{t_k-1}}{(E_k : E_k \cap {\rm N}_{K/k} (K^\times))}; $$

one sees some similarities between this formula and the inequality
given by Proposition \ref{propo2},
but we can not deduce a more efficient inequality on the ranks
because we only have
${\rm rk}_p (\Cl_K^G) \leq \ds\frac{{\rm log}(\order \Cl_k)}{{\rm log}(p)} + t_k - 1$,
and the order of $\Cl_k$ may be huge even if its $p$-rank may be evaluated.

\smallskip
The Chevalley formula gives a more precise inequality when $\order \Cl_k$ is small
(e.g., $\Cl_k=1$, giving ${\rm rk}_p (\Cl_K^G) \leq  t_k - 1$); on the contrary, our 
inductive method controls more the $p$-ranks than the orders.}
\end{remark}

Denote by $\{\ell_1, \ldots , \ell_N\}$, $N=N_{i-1} \geq 0$, the set of tame
prime numbers ramified in $K/k = F_i/F_{i-1}$ (for such a $\ell$,
there exist prime ideals ${\mathfrak l}_{u} \mid \ell$ in $k$,
$u \in [1, t_{k,\ell}]$, $t_{k,\ell} \geq 1$, ramified in $K/k$ and
similarly for $p$). Thus, with the above notations:

\smallskip
\centerline{$t_k = t_{k,p} + \sm_{j=1}^{N} t_{k,\ell_j}^{}$.}

\smallskip
We shall replace, in Proposition \ref{propo2},  
${\rm rk}_p ( E_k/E_k^p)$, then $t_{k,p}$ and $t_{k,\ell_j}^{}$, by the 
rough upper bound $d_k := [k : \Q] = p^{i-1} \, [\Kappa : \Q]$, which gives, using 
inequality (\ref{inequality}), Corollary \ref{coro0} and Proposition \ref{propo2}:

\begin{proposition} \label{propo3} Put $d_k = [k : \Q]$ and let
$N$ be the number of {\it tame} prime numbers ramified in $K/k$ (cyclic of degree $p$); 
we have the inequalities:

\medskip
\centerline{\hspace{0.55cm}${\rm rk}_p(\Cl_K^*) \leq (p-1) \cdot \big [\, {\rm rk}_p(\Cl_k) + 
(N+2)\,d_k \, \big ]$,}

\medskip
\centerline{\hspace{0.5cm}${\rm rk}_p(\Cl_K) \leq p \cdot {\rm rk}_p(\Cl_k) + 
(p-1)\, (N+2)\,d_k$.}
\end{proposition}

\section{About the discriminants in the tower $(F_i)_{0 \leq i \leq e}$}
Now, the goal is to give lower bounds for discriminants, unlike for the 
case of $p$-ranks. Give some essential explanations:

\begin{remark}\label{rema}{\rm
We have given, in the previous section, an upper bound for the 
$p$-rank of $\Cl_K$ in $K/k=F_i/F_{i-1}$, where the number of ramified primes is 
crucial because of ``genus theory'' aspects, so that {\it any true $p$-rank} 
at the step $K/k$ (whatever $F \in {\mathscr F}_{\!\!\kappa}^{p^e}$) 
will be {\it smaller}\,; note that each integer $N=N_{i-1}$ is relative to 
the step $K/k=F_i/F_{i-1}$, which will be also the case of the relative 
discriminants $D_{K/k}$.

\smallskip
In the present section, we shall give a definition of the 
``minimal tame relative discriminant'' $D_N^{\rm ta} \in \N$ 
only depending on $N$, then a minor modification giving
$D_N \geq D_N^{\rm ta}$ (so that any {\it true} relative 
discriminant $D_{K/k} = D_{K/k}^{\rm ta} \times \hbox{($p$-power)} \in \N$ 
(from ramification of tame $\ell_j$ and possibly that of $p$), will be 
{\it larger} than $D_N$). Then we shall apply the relation
$D_K = D_k^p\, D_{K/k} \geq D_k^p\, D_N$.

\smallskip
Thus, under this facts, if an ``$\varepsilon$-inequality'' 
$\order (\Cl_K \otimes \F_p) \ll (\sqrt{D_k^p\, D_N}\,)^\varepsilon$
does exist between such fictitious $p$-ranks and discriminants, a fortiori, any 
effective situation will fulfill the $p$-rank $\varepsilon$-inequality at this step.}
\end{remark}

\subsection{Tame relative discriminant -- Minimal relative discriminant}

As above, let $k=F_{i-1}$, $K=F_i$, $N=N_{i-1}$
and let $D^{\rm ta}_k$ and $D^{\rm ta}_K$ be the absolute values 
of {\it the tame parts} of the discriminants of $k$ and $K$, respectively.
For the  $N$ tame prime numbers $\ell$, ramified in $K/k$, let
${\mathfrak l}_1, \ldots, {\mathfrak l}_{t_{k,\ell}}$ be the 
prime ideals of $k$ above $\ell$, ramified in $K/k$.

\smallskip
The relative norm of the different of $K/k$ gives the tame part 
of the relative ideal discriminant of $K/k$, namely 
${\mathcal D}^{\rm ta}_{K/k} = \prd_{j=1}^N  
\prd_{u=1}^{t_{k,\ell_j}} {\mathfrak l}_{j,u}^{p-1}$, and its absolute 
norm is $D_{K/k}^{\rm ta}$ (tame relative discriminant); the tame 
``discriminant formula'' \cite[Propositions IV.4 and III.8]{Se} yields 
($f_{k,{\mathfrak l}_{j,u}}$ is the residue degree of ${\mathfrak l}_{j,u}$ in $k/\Q$):
\begin{equation*}
D^{\rm ta}_K = (D^{\rm ta}_k)^p \cdot {\rm N}_{k/\Q}({\mathcal D}^{\rm ta}_{K/k})= 
(D^{\rm ta}_k)^p \cdot D^{\rm ta}_{K/k} \leq D_k^p \cdot D^{\rm ta}_{K/k},
\end{equation*}
where $D^{\rm ta}_{K/k} = \prd_{j=1}^N  \prd_{u=1}^{t_{k,\ell_j}} 
\ell_j^{(p-1)\, f_{k,{\mathfrak l}_{j,u}}} = \prd_{j=1}^N \,\ell_j^{(p-1) \cdot 
\sum_{u=1}^{t_{k,\ell_j}} f_{k,{\mathfrak l}_{j,u}}}$.

\medskip
We intend now to determine a lower 
bound $D_N$ of $D_{K/k}$, only depending on $N$,
so that for every concrete ramification in $K/k$, with $N$ tame 
primes $\ell_j$ and possibly that of $p$, the effective relative 
discriminant $D_{K/k}$ will be necessarily larger 
than $D_N$ as explained in Remark \ref{rema}. 
This needs two obvious lemmas.

\begin{lemma}\label{lem1}
A lower bound of the tame relative discriminant $D^{\rm ta}_{K/k}$ is
obtained when $t_{k,\ell_j} = f_{k,{\mathfrak l}_{j}} = 1$ for all $j=1,\ldots,N$; 
thus the ``fictive'' minimum of $D^{\rm ta}_{K/k}$ is then 
$D^{\rm ta}_N := \prod_{j=1}^N q_j{}^{p-1}$, 
taking the $N$ successive prime numbers $q_j\ne p$.
\end{lemma}

We may call $D^{\rm ta}_N$ the {\it minimal 
tame relative discriminant} (for instance, for $p=2$, $D^{\rm ta}_N = 
3 \cdot 5 \cdot 7 \cdots q_N$); whence, an analogous 
framework as for the case of degree $p$ 
cyclic number fields given in \cite{Gr00}.
Note that, when $\Kappa \ne \Q$ the primes $\ell$, ramified in $K/k$,
are not necessarily such that $\ell \equiv 1$ (mod $p$) (e.g., $p=3$,  
$\Kappa = \Q(\mu_3)$ and $F= \Kappa(\sqrt[3^e\!]{2})$).

\medskip
The second lemma will simplify the forthcoming computations:

\begin{lemma}\label{lem11}
One can replace $D^{\rm ta}_N = \prod_{j=1,\,q_j \ne p}^N q_j{}^{p-1}$ by
$D_N = \prod_{j=1}^N q_j{}^{p-1}$.
\end{lemma}

\proof 
The claim is obvious if $p>q_N$. Otherwise, there are 
two cases for a true relative discriminant $D_{K/k}$:

\smallskip
(i) $p \mid D_{K/k}$. In this case the $p$-part of $D_{K/k}$ is a $p$-power 
larger than $p$ and the required inequality between $D_{K/k}$ and $D_N$ holds;

\smallskip
(ii) $p \nmid D_{K/k}$. In this case, since $p \mid D_N$ and since $D_{K/k}$ has $N$ 
tame ramified primes, there exists at least an $\ell \mid D_{K/k}$ (``replacing'' $p$), 
$\ell \nmid D_N$; so this prime is larger than $q_N \geq p$, thus $D_{K/k}>D_N$.
\qed

\subsection{Induction}
Let $F \in {\mathscr F}_{\!\!\kappa}^{p^e}$. The case $i=0$ ($k=\Kappa$) 
is obvious since to get, for all $\varepsilon > 0$:
$$\order (\Cl_\kappa \otimes \F_p) = p^{{\rm rk}_p(\Cl_\kappa)} 
\leq C_{\kappa, p, \varepsilon} 
\cdot (\sqrt {D_\kappa}\,)^\varepsilon, $$ 
it suffices to take $C_{\kappa, p, \varepsilon} = 
C_{\kappa, p} := p^{{\rm rk}_p(\Cl_\kappa)}$ since 
$(\sqrt {D_\kappa}\,)^\varepsilon > 1$; in other words one may replace 
${\mathscr F}_{\!\!\kappa}^{p^e}$ by the family${\mathscr F}_{\!\!d,\rho}^{p^e}$
where $\Kappa$, of degree $d_\kappa \leq d$, varies under the condition 
${\rm rk}_p(\Cl_\kappa) \leq \rho$, $d$, $\rho$ given (e.g., one may consider all the 
$p$-principal base fields $\Kappa$ of degree $d_\kappa \leq d$ with $\rho=0$).

\smallskip
By induction, we assume (for $k:=F_{i-1}$, $K=F_i$) that for all $\varepsilon >0$ 
there exists a constant $C_{k, p, \varepsilon}$ such that
$p^{{\rm rk}_p(\Cl_k)} \leq C_{k, p, \varepsilon} \cdot (\sqrt {D_k}\,)^\varepsilon$
independently of the number $N$ of tame ramified primes $\ell_j$
in $K/k$; then we shall prove the property for $K$.

\smallskip
Proposition \ref{propo3} implies the following inequality (where $d_k :=[k : \Q]$):
\begin{equation}\label{eq}
 p^{{\rm rk}_p(\Cl_K)} \leq p^{p\, \cdot \, {\rm rk}_p(\Cl_k) + 
(p-1)\,(N+2)\,d_k}  
\leq C_{k, p, \varepsilon}^p\! \cdot \hbox{$\big( \sqrt {D_k^p}\,\big)^{\varepsilon}$} \!\!
\cdot p^{(p-1)\,(N+2)\,d_k}.
\end{equation}

\noindent
Then, as explained in Remark \ref{rema}, we will compare $p^{(p-1) \cdot (N+2)\,d_k}$ 
and $\big(\sqrt{D_N}\,\big)^{\varepsilon}$, where (Lemmas \ref{lem1}, \ref{lem11})
$D_N=\prd_{j=1}^N q_j{}^{p-1}$, the $q_j$ being the $N$ consecutive prime 
numbers whatever $p$, then using the inequality $D_K \geq D_k^p D_N$.

\smallskip
But we have the required computations in \cite[\S\,2.3, Formulas (4, 5), p.\,10]{Gr00} 
that we improve with some obvious modifications. 

\begin{proposition} \label{c}
There exists $c_{k, p, \varepsilon}$ such that
$p^{(p-1) (N+2)\,d_k} \leq 
c_{k, p, \varepsilon} \cdot \big(\sqrt{D_N}\,\big)^{\varepsilon}$.
\end{proposition}
   
\proof
For $N=0$, $D_N=1$ ($K/k$ is at most $p$-ramified), so that the result 
(independent of $\varepsilon $) is true since the constant $c_{k, p, \varepsilon}$,
resulting of the computation of a maximum taken over $N \in \N$ 
(see Remark \ref{max}), will be much larger than $p^{2\, (p-1) \, d_k}$
as we shall verify; we assume $N \geq 1$ in what follows. 

\smallskip
The existence of $c_{k, p, \varepsilon}$ is equivalent to the fact that
$p^{(p-1) (N+2)\, d_k}  (\sqrt{D_N}\,)^{-\varepsilon}$ is bounded over $N$, 
whence $(p-1) (N+2)\,d_k \,{\rm log}(p) - \varepsilon \cdot 
\hbox{$\frac{p-1}{2}$} \!\! \sm_{j=1}^N {\rm log}(q_j) < \infty$,
in which case, ${\rm log}(c_{k, p, \varepsilon})$ is given by the
upper bound.

\smallskip
The main purpose is to estimate the sum $\sum_{j=1}^N {\rm log}(q_j)$,
knowing that we can replace this sum by any convenient lower bound but
noting that this will increase the constant $c_{k, p, \varepsilon}$.

\smallskip
We replace the consecutive primes $q_j$ (except for $q_1=2$) by the lower bounds:

\medskip
\centerline{$q'_1=2$,  \ \,  $q'_j = \hbox{$\frac{1}{2}$} \, j \,
{\rm log}\big (\hbox{$\frac{q_j}{2}$} \big)$,  $j \geq 2$}

\medskip\noindent
(cf. \cite[Notes on Ch. I, \S\,4.6]{T} about the Montgomery--Vaughan result giving, 
for all prime numbers, this lower bound without error term). Thus a sufficient condition
to have $p^{(p-1) (N+2)\, d_k} \cdot (\sqrt{D_N}\,)^{-\varepsilon}$ bounded is that:

\medskip
\leftline{$X(N) :=  (p-1) (N+2)\, d_k\, {\rm log}(p)$}

\rightline{$ - \varepsilon \cdot \hbox{$\frac{p-1}{2}$}
\Big({\rm log}(2) +\sm_{j=2}^N \Big[ {\rm log} 
\big(\hbox{$\frac{1}{2}$}) +  {\rm log} (j)
+ {\rm log}_2 \big (\hbox{$\frac{q_j}{2}$} \big) \Big]\Big) < \infty$.}

We check that $e := {\rm log}(2)+
\sm_{j=2}^N \Big[ {\rm log}(\hbox{$\frac{1}{2}$}) + 
{\rm log}_2 \big(\hbox{$\frac{q_j}{2}$} \big) \Big]$
is positive, except few values of $N$, but that this does not contradict
the whole lower bound, as shown by the following PARI/GP \cite{P}
program computing $E=S-t$ and $e=s-t$, with:

\centerline{$S=\sm_{j=1}^N {\rm log}(q_j)$,
$s={\rm log}(2) +\sm_{j=2}^N \Big[ {\rm log} 
\big(\hbox{$\frac{1}{2}$}) +  {\rm log} (j)
+ {\rm log}_2 \big (\hbox{$\frac{q_j}{2}$} \big) \Big]$,
$t=\sm_{j=2}^N {\rm log}(j)$:}

\smallskip\noindent
thus $E$ measures the 
approximation regarding the lower bound $t$ of $S$.

\footnotesize
\begin{verbatim}
{for(N=1,100,S=log(2);s=log(2);t=0;for(j=2,N,el=prime(j);
S=S+log(el);s=s+log(j/2)+log(log(el/2));t=t+log(j));E=S-t;e=s-t;
print("N=",N," E=S-t=",precision(E,1)," e=s-t=",precision(e,1));
print("S=",precision(S,1)," s=",precision(s,1)," t=",precision(t,1)))}

N=1   E=S-t=0.693147180559945309  e=s-t=0.6931471805599453095
      S=0.693147180559945309  s=0.693147180559945309 t=0
N=2   E=S-t=1.098612288668109691  e=s-t=-0.902720455717879982
      S=1.791759469228055000  s=-0.20957327515793467 t=0.693147180559945309
N=3   E=S-t=1.609437912434100374  e=s-t=-1.683289208068580387
      S=3.401197381662155376  s=0.108470261159474613 t=1.791759469228055000
N=4   E=S-t=2.169053700369523061  e=s-t=-2.151084901802564193
      S=5.347107530717468681  s=1.026968928545381426 t=3.178053830347945620
N=5   E=S-t=2.957511060733793231  e=s-t=-2.310814729030344016
      S=7.745002803515839225  s=2.476677013751701978 t=4.787491742782045994
N=6   E=S-t=3.730700948967274966  e=s-t=-2.377060211850912773
      S=10.30995216097737596  s=4.20219100015918822  t=6.579251212010100996
N=7   E=S-t=4.618004143968177741  e=s-t=-2.309370646333412833
      S=13.143165505033592041 s=6.21579071473200146  t=8.525161361065414300
N=8   E=S-t=5.483001581454782273  e=s-t=-2.191013642714161849
      S=16.087604484200032501 s=8.41358926003108838  t=10.60460290274525022
N=9   E=S-t=6.421271220047712581  e=s-t=-1.9912013465566233694
      S=19.223098700129182192 s=10.8106261335248462  t=12.80182748008146961
N=10  E=S-t=7.485981957040140924  e=s-t=-1.7007174593212892239
      S=22.590394530115656220 s=13.4036951137542260  t=15.10441257307551529
N=11  E=S-t=8.522073888726916626  e=s-t=-1.3856001883918199308
      S=26.024381734600802466 s=16.1167076574820659  t=17.50230784587388584
N=12  E=S-t=9.64808515158314076   e=s-t=-1.0079274921626889861
      S=29.63529964724502690  s=18.9792870034991971  t=19.98721449566188615
N=13  E=S-t=10.7967078608259118   e=s-t=-0.5956771604555961438
      S=33.34887171394933471  s=21.9564866926678267  t=22.55216385312342288
N=14  E=S-t=11.9188506469042156   e=s-t=-0.16778120370448471046
      S=37.11007182964289714  s=25.0234399790341967  t=25.19122118273868150
N=15  E=S-t=13.0609480475120641   e=s-t=0.2886939587134154542
      S=40.96021943135295572  s=28.1879653425543070  t=27.89927138384089156
N=16  E=S-t=14.2586512388244047   e=s-t=0.7825193132031076079
      S=44.93051134490507756  s=31.4543794192837804  t=30.67186010608067280
N=17  E=S-t=15.5029753386739081   e=s-t=1.3085458937854073287
      S=49.00804878881079701  s=34.8136193439222962  t=33.50507345013688888
N=18  E=S-t=16.7234774449510546   e=s-t=1.8443743308869917078
      S=53.11892265298410826  s=38.2398195389200452  t=36.39544520803305358
N=19  E=S-t=17.9837310851755802   e=s-t=2.407283386866249599
      S=57.32361527237507432  s=41.7471675740657436  t=39.33988418719949404
N=20  E=S-t=19.250678688662904708 e=s-t=2.986570896185892937
      S=61.58629514941638974  s=45.3221873569393779  t=42.33561646075348503
N=21  E=S-t=20.496615692087872840 e=s-t=3.573610687966718782
      S=65.87675459056478087  s=48.9537495864436268  t=45.38013889847690803
N=22  E=S-t=21.775021091196578482 e=s-t=4.182370501783588325
      S=70.24620244303180236  s=52.6535518536188122  t=48.47118135183522388
N=23  E=S-t=23.058367483064026714 e=s-t=4.804476310693340535
      S=74.66504305082840029  s=56.4111518784577141  t=51.60667556776437357
N=24  E=S-t=24.368950022448220934 e=s-t=5.445142436276319277
      S=79.15367942056054013  s=60.2298718343886384  t=54.78472939811231919
(...)
N=99  E=S-t=140.388608477587341   e=s-t=72.45660863104452031
      S=499.5228138471627406  s=431.5908140006199191 t=359.1342053695753988
N=100 E=S-t=142.07685757044573    e=s-t=73.48627663602501092
      S=505.8162331260092221  s=437.2256521915885011 t=363.7393755555634902
(...)
N=10000 E=S-t=22283.274179035430378  e=s-t=15748.9389769817289
        S=104392.20201584978383 s=97857.86681379608238  t=82108.9278368143
N=10001 E=S-t=22285.623003678229650  e=s-t=15750.6314792930841
        S=104403.76128085955962 s=97868.76975647441412  t=82118.1382771813
\end{verbatim}

\normalsize
Thus, since $E=S-t$ is always positive, we may consider, instead:

\smallskip
\centerline{$X(N)=(p-1) (N+2)\,d_k \,{\rm log}(p)
 - \varepsilon  \cdot \hbox{$\frac{p-1}{2}$}\,\sm_{j=1}^N  {\rm log} (j)$}
 
\smallskip\noindent
for which it suffices to prove that $X(N)< \infty$.
We have, for all $N \geq 1$, ${\rm log}(N!) = N {\rm log}(N)  - N 
+ \hbox{$\frac{{\rm log}(N)}{2}$}  + 1 - o(1)$, giving:

\medskip
\centerline {$X(N) =(p-1) (N+2)\,d_k\, {\rm log}(p) 
 - \varepsilon \!\cdot\! \hbox{$\frac{p-1}{2}$} \Big[ N {\rm log}(N)
-N  + \hbox{$\frac{{\rm log}(N)}{2}$}  + 1 - o(1) \Big]$,}

\medskip\noindent
whence:
$$X(N)= - \varepsilon \, \hbox{$\frac{p-1}{2}$}\,N\, {\rm log}(N) +
N\, \Big[(p-1)\,d_k \, {\rm log}(p)+ \varepsilon \,\hbox{$\frac{p-1}{2}$} 
- \varepsilon \,\hbox{$\frac{p-1}{2}$} \, o(1) \Big ], $$
where the new $o(1)$ is of the form 
$\frac{{\rm log}(N)}{2 N} +\frac{1-o(1)}{N} > 0$. The dominant term:
$$- \varepsilon \, \hbox{$\frac{p-1}{2}$}\,N\, {\rm log}(N), \ N \geq 1,$$
ensures the existence of a bound $c_{k, p, \varepsilon}$ 
over $N  \geq 1$, $N \to \infty$.\qed

\begin{remark}\label{max}{\rm
To obtain an explicit upper bound of the constant $c_{k, p, \varepsilon}$,
one may replace the function $X(N)$ by $Y(N) > X(N)$ defined by:
$$Y(N):=  - \varepsilon \, \hbox{$\frac{p-1}{2}$}\,N\, {\rm log}(N) +
N\, \Big[(p-1)\,d_k \, {\rm log}(p)+ \varepsilon \,\hbox{$\frac{p-1}{2}$}) \Big ]. $$

One obtains that $Y(N)$ admits, for 
an $N_0 (\varepsilon)> 0$, a maximum given by
$${\rm log}(N_0 (\varepsilon)) = 
2\,d_k\, {\rm log}(p)\cdot \varepsilon^{-1}, $$

which yields a maximum for ${\rm log}(c_{k, p, \varepsilon})$ less than:
$$\hbox{$\frac{p-1}{2} \, \varepsilon \cdot 
e_{}^{2\,d_k \, {\rm log}(p)\,\cdot\, \varepsilon^{-1}}$,}$$
giving an important constant.
This is due in part to the method using certain extreme bounds
which are not achieved in practice (for example the systematic use 
of the upper bound $d_k=[k : \Q]$ to get Proposition \ref{propo3}). }
\end{remark}

Finally, from formula (\ref{eq}) and Proposition \ref{c} giving:

\medskip
\centerline{$p^{(p-1) (N +2)\, d_k} \leq c_{k, p, \varepsilon} \cdot 
(\sqrt{D_N}\,)^{\varepsilon} \leq c_{k, p, \varepsilon} \cdot 
(\sqrt{D_{K/k}}\,)^{\varepsilon}$,}

\medskip\noindent
we may write in $K/k$:

\medskip
\leftline{$p^{{\rm rk}_p(\Cl_K)} \leq C_{k, p, \varepsilon}^p \cdot
\hbox{$\big(\sqrt {D_k^p}\,\big)^{\varepsilon}$}
\cdot \, p^{(p-1) (N +2)\, d_k}$}

\medskip
\centerline{\hspace{1.1cm} $\leq C_{k, p, \varepsilon}^p \cdot
\hbox{$\big(\sqrt {D_k^p}\,\big)^{\varepsilon}$} \cdot c_{k, p, \varepsilon} \cdot
\hbox{$\big(\sqrt {D_{K/k}}\,\big)^{\varepsilon}$} 
=  C_{k, p, \varepsilon}^p \cdot c_{k, p, \varepsilon} \cdot 
\hbox{$\big(\sqrt {D_k^p \cdot D_{K/k}}\,\big)^{\varepsilon}$} $}

\medskip
\leftline{\hspace{1.45cm}$= C_{K, p, \varepsilon} \cdot 
(\sqrt{D_K}\,)^{\varepsilon}, $}

\medskip\noindent
with $C_{K,p,\varepsilon} := C_{k, p, \varepsilon}^p \cdot c_{k, p, \varepsilon}$.
The degree $[F : \Kappa] = p^e$ being fixed, the above induction leads to
(denoting $C_{F,p,\varepsilon} =: C_{\kappa ,p^e,\varepsilon}$):

\begin{theorem}\label{thmf}
Let $p \geq 2$ be prime. The $p$-rank $\varepsilon$-conjecture for the
family ${\mathscr F}_{\!\!\kappa}^{p^e}$ of $p$-cyclic-towers $F/\Kappa$
of degree $p^e$, on the existence, for all $\varepsilon > 0$, of a constant 
$C_{\kappa ,p^e,\varepsilon}$ such that 
$\order (\Cl_F \otimes \F_p) \leq C_{\kappa ,p^e,\varepsilon}
(\sqrt{D_F}\,)^\varepsilon$, is fulfilled unconditionally
for all $F \in {\mathscr F}_{\!\!\kappa}^{p^e}$.
\end{theorem}

\begin{remark} \label{Rank}
{\rm Let $r \geq 1$. It is obvious that, using the filtration 
$(M^*_h)_{h \geq 0}$, analogous computations with
$M^*[p^r] := \{x \in M^*, \ x^{p^r}=1\}$ give (from \eqref{ineqfond}):
$$\order M^*[p^r] = M^*_{r\,(p-1)} \leq (\order M^*_1)^{r \, (p-1)}, $$
then, for $M^*=\Cl_K^*$ and ${\rm rk}_p(\Cl_K^{*\,G}) \leq {\rm rk}_p(\Cl_k) + (N+2)\,d_k$ 
(Proposition \ref{propo2}), written under the form
$(\order \Cl_K^{* \,G})^{r\,(p-1)} \leq 
p^{r\,(p-1)\, {\rm rk}_p(\Cl_k) + r\,(p-1)\, (N+2)\,d_k}$, we get:
$$\order (\Cl_K^*\otimes \Z/p^r \Z) \leq 
p^{r\,(p-1)\,{\rm rk}_p(\Cl_k) + r\,(p-1)\, (N+2)\,d_k}. $$

Whence, using
$1 \to \Cl_K^*\otimes \Z/p^r \Z \to \Cl_K \otimes \Z/p^r \Z
\to \Cl_k \otimes \Z/p^r \Z$, and since 
$\Cl_k \otimes \Z/p^r \Z \leq p^{r \,{\rm rk}_p(\Cl_k)}$, we obtain:
$$\order (\Cl_K \otimes \Z/p^r \Z) \leq 
p^{r\,p\,{\rm rk}_p(\Cl_k) + r\,(p-1)\, (N+2)\,d_k}. $$}
\end{remark}

Finally, we may consider the following statement as a consequence
of the above computations for $r=1$:

\begin{corollary} \label{CRank}
Let $r \geq 1$ be a fixed integer. Then, for all $\varepsilon > 0$, 
there exists a constant $C^{(r)}_{\kappa ,p^e,\varepsilon}$ such that 
$\order (\Cl_F \otimes \Z/p^r\Z) \leq C^{(r)}_{\kappa ,p^e,\varepsilon}
(\sqrt{D_F}\,)^\varepsilon$.
\end{corollary}

\begin{proof} The introduction of $r$ changes:
$$X(N) = (p-1) (N+2)\,d_k \,{\rm log}(p)
 - \varepsilon  \cdot \hbox{$\frac{p-1}{2}$}\,\sm_{j=1}^N  {\rm log} (j)$$ 
into the expression:
$$X^{(r)}(N) = r\,(p-1) (N+2)\,d_k\, {\rm log}(p)  - \varepsilon  
\cdot \hbox{$\frac{p-1}{2}$}\,\sm_{j=1}^N  {\rm log} (j), $$ 

which does not modify the dominant term 
$- \varepsilon \, \hbox{$\frac{p-1}{2}$}\,N\, {\rm log}(N)$
coming from the right term.
Then the induction is similar with the exponent $r$ which is a constant
and does not modify the reasoning (we omit the details).
\end{proof}

\begin{remark} {\rm
 If $F \in {\mathscr F}_{\!\!\kappa}^{p^e}$ is contained in 
the $p$-Hilbert tower of $\Kappa$ (which defines a sub-family ${\mathscr F}'^{p^e}_{\!\!\kappa}$
of $p$-cyclic towers when $F$ varies, especially when the $p$-Hilbert tower is infinite), 
we have $D_F= D_\kappa ^{p^e}$, whence:
$$\order (\Cl_F \otimes \F_p) \leq 
C_{\kappa ,p^e,\varepsilon} \cdot (\sqrt{D_\kappa}\,)^{\varepsilon \,p^e}; $$
 
renormalizing $\varepsilon$, since $p^e$ is a constant, we
may write: for all $\varepsilon > 0$, there exists a constant 
$C'_{\kappa ,p^e,\varepsilon}$ such that, for all such unramified 
$p$-towers $F \in {\mathscr F}'^{p^e}_{\!\!\kappa}$:
$$\order (\Cl_F \otimes \F_p) \leq C'_{\kappa ,p^e,\varepsilon} \cdot 
(\sqrt{D_\kappa}\,)^{\varepsilon}. $$
For other approaches about $p$-ranks in towers as the degree grows, 
see for instance Hajir \cite{Haj} and Hajir--Maire \cite{HM}.}
\end{remark}

\section{The $p$-rank $\varepsilon$-conjecture in $p$-ramification theory}
\label{sec4}

We shall replace, for the family 
${\mathscr F}_{\!\!\kappa}^{p^e}$, the $p$-class group $\Cl_F$ by the Galois group 
${\mathcal A}_F$ of the maximal $p$-ramified
abelian pro-$p$-extension ${\mathcal H}_F^{\rm ab}$ of $F$ or its
torsion group ${\mathcal T}_F$ \cite[III.2, IV.3]{Gr0}.
As we know, this pro-$p$-group ${\mathcal A}_F$ is a fundamental invariant 
related to the $p$-class group $\Cl_F$, the normalized $p$-adic regulator 
${\mathcal R}_F$ and the number of independent $\Z_p$-extensions of $F$
($\Z_p$-rank of ${\mathcal A}_F$ depending on 
Leopoldt's conjecture); see, e.g., \cite[IV,\S\S1,2,3]{Gr0}, \cite{Gr3} and the 
very complete bibliography of \cite{Gr1} for the story of abelian 
$p$-ramification theory, especially the items 
[3, 16, 17, 18, 19, 26, 40, 50, 57, 58, 59, 63, 65, 67, 70, 72].

\smallskip
Give some recalls about ${\mathcal A}_F$, ${\mathcal T}_F$ and the
corresponding fixed point formulas.
We assume the Leopoldt conjecture for $p$ in all the fields considered.

\medskip
Let ${\mathcal G}_F$ be the Galois group of the maximal pro-$p$-extension 
${\mathcal H}_F$ of $F$, $p$-ramified (i.e., unramified outside $p$ and 
non-complexified (= totally split) at the real infinite places of $F$ when $p=2$).

\smallskip
Its abelianized ${\mathcal A}_F = {\rm Gal}({\mathcal H}_F^{\rm ab}/F)$ 
is a $\Z_p$-module of finite type for which ${\mathcal T}_F := 
{\rm tor}_{\Z_p}({\mathcal A}_F)$, isomorphic to the dual of the 
cohomology group ${\rm H}^2({\mathcal G}_F,\Z_p)$ \cite{Ng},  fixes the 
compositum $\wt F$ of the $\Z_p$-extensions of $F$.
Then:
$${\mathcal A}_F \simeq \Z_p^{r_2^{}(F)+1} \plus {\mathcal T}_F, $$

where $r_2^{}(F)$ is the number of complex embeddings of $F$.

\smallskip
Let $U_F := \bigoplus_{{\mathfrak p} \mid p} U_{\mathfrak p}$, be the product of 
the principal local units of the completions $F_{\mathfrak p}$ of $F$ at the $p$-places,
let $\ov {E}_F$ be the closure in $U_F$ of the diagonal image of the group $E_F$ of global units 
and let 
$${\mathcal W}_F := {\rm tor}_{\Z_p}(U_F) \big/{\rm tor}_{\Z_p}(\ov {E}_F) =
{\rm tor}_{\Z_p}(U_F) \big/ \mu_p(F)$$ 

under Leopoldt's conjecture. 

\smallskip
Then $U_{F}/\ov {E}_{\!F}$ (resp. ${\mathcal W}_F$) fixes the $p$-Hilbert class 
field $H_F$ (resp.  the Bertran\-dias--Payan field ${\mathcal H}_F^{\rm bp}$) and
${\mathcal R}_F$ is the normalized $p$-adic regulator of $F$ (classical 
$p$-adic regulator up to an obvious $p$-power, cf. \cite[Proposition 5.2]{Gr3}):

\unitlength=1.0cm 
\vbox{\hbox{\begin{picture}(10.5,6.2)
\put(6.5,4.50){\line(1,0){1.3}}
\put(8.7,4.50){\line(1,0){2.1}}
\put(3.85,4.50){\line(1,0){1.4}}
\put(9.2,4.11){\footnotesize $\simeq\! {\mathcal W}_F$}
\put(4.4,2.50){\line(1,0){1.0}}
\bezier{550}(3.8,4.9)(7.6,5.5)(10.8,4.9)
\put(6.2,5.4){\footnotesize ${\mathcal T}_F 
\simeq {\rm H}^2({\mathcal G}_F,\Z_p)^\ast$}
\put(3.50,2.9){\line(0,1){1.25}}
\put(3.50,0.9){\line(0,1){1.25}}
\put(5.7,2.9){\line(0,1){1.25}}
\bezier{300}(3.85,0.5)(4.8,0.8)(5.6,2.3)
\put(5.2,1.3){\footnotesize $\simeq \! \Cl_F$}
\bezier{300}(6.3,2.5)(8.5,2.6)(10.8,4.3)
\put(7.9,2.5){\footnotesize $\simeq \! U_{F}/\ov {E}_{\!F}$}
\put(10.85,4.4){${\mathcal H}_F^{\rm ab}$}
\put(5.4,4.4){$\wt F\!H_F$}
\put(7.85,4.4){${\mathcal H}_F^{\rm bp}$}
\put(6.7,4.10){\footnotesize $\simeq\! {\mathcal R}_F$}
\put(3.3,4.4){$\wt F$}
\put(5.5,2.4){$H_F$}
\put(2.9,2.4){$\wt F \!\cap \! H_F$}
\put(3.4,0.40){$F$}
\put(8.9,1.5){\footnotesize ${\mathcal A}_F$}
\bezier{500}(3.85,0.4)(9.5,0.8)(11.0,4.3)
\end{picture} }}

Let $K/k$ be {\it any extension} of number fields; then from \cite[Theorem IV.2.1]{Gr0} 
the transfer map ${\rm J} : {\mathcal A}_k \too {\mathcal A}_K$ is always injective
under Leopoldt's conjecture. This will be applied to the degree 
$p$ cyclic sub-extensions of a tower $F/\Kappa$. 

\smallskip
The analogue of the exact sequence (\ref{clg}) for class groups, 
used in the proof of Proposition \ref{propo2}, is given by the following 
result \cite[Proposition IV.3.2.1]{Gr0}:

\begin{proposition}
Let $K/k$ be any Galois extension of number fields and let $G = {\rm Gal}(K/k)$.
We have, under Leopoldt's conjecture, the exact sequence:

\smallskip
\centerline{$1 \too {\rm J}({\mathcal A}_k) \simeq {\mathcal A}_k 
\tooo {\mathcal A}_K^G \tooo 
\plus_{{\mathfrak l}_k \nmid\, p} \, \Z_p/ e_{{\mathfrak l}_k} \Z_p \too 0$,}

\smallskip\noindent
$ e_{{\mathfrak l}_k}$ being the ramification index of the 
prime ideals ${\mathfrak l}_k \nmid p$ of $k$ ramified in $K/k$.
\end{proposition}

\begin{corollary}\label{coroA}
We have ${\rm rk}_p({\mathcal A}_K^G) \leq
{\rm rk}_p({\mathcal A}_k) + t^{\rm ta}_k$, where $t^{\rm ta}_k$ is the number
of prime ideals ${\mathfrak l}_k \nmid p$ of $k$ ramified in $K/k$.
\end{corollary}

Consider the framework of degree $p$ cyclic 
extensions $K/k = F_i/F_{i-1}$ related to a $p$-cyclic tower 
$F \in {\mathscr F}_{\!\!\kappa}^{p^e}$. 
Let ${\mathcal A}_K^* = {\rm Ker}(\Nu)$, with $\Nu =
{\rm J} \circ {\rm N}$, where ${\rm N} : {\mathcal A}_K \to {\mathcal A}_k$ is 
the restriction of automorphisms; we have similarly:
$$\hbox{${\rm rk}_p({\mathcal A}_K) \leq {\rm rk}_p({\mathcal A}_k) +
{\rm rk}_p({\mathcal A}_K^*)$ and
${\rm rk}_p({\mathcal T}_K) \leq {\rm rk}_p({\mathcal T}_k) +
{\rm rk}_p({\mathcal T}_K^*)$.} $$

Then considering Proposition \ref{propo1} (valid for the $\Z_p$-modules of finite
type $M={\mathcal A}_K$ since the $M^*_{h+1}/M^*_h$ are elementary finite
$p$-groups for all $h \geq 0$), one obtains from Corollary \ref{coroA} in $K/k$:
$${\rm rk}_p({\mathcal A}_K^*) \leq 
(p-1)\,{\rm rk}_p({\mathcal A}_K^{G*}) = 
(p-1)\,{\rm rk}_p({\mathcal A}_K^G) \leq 
(p-1)\, ({\rm rk}_p({\mathcal A}_k) + t^{\rm ta}_k), $$
whence:
$${\rm rk}_p({\mathcal A}_K) \leq p\,{\rm rk}_p({\mathcal A}_k)+(p-1)\, t^{\rm ta}_k. $$

Let $N$ be the number of tame primes $\ell_j$, ramified in $K/k$;
using the same upper bound $d_k := [k : \Q]$, we obtain:
\begin{equation}\label{r22}
{\rm rk}_p({\mathcal A}_K) \leq p\,{\rm rk}_p({\mathcal A}_k)+(p-1)\,N\,d_k. 
\end{equation}

\begin{theorem}\label{thmf2}
Let $p \geq 2$ be a prime number. Under Leopoldt's conjecture for~$p$, 
the $p$-rank $\varepsilon$-conjectures: 
$$\order ({\mathcal A}_F \otimes \F_p)
\ll_{\kappa ,p^e,\varepsilon} (\sqrt{D_F}\,)^\varepsilon \  \hbox{ and }\ 
\order ({\rm H}^2({\mathcal G}_F,\Z_p) \otimes \F_p)
\ll_{\kappa ,p^e,\varepsilon} (\sqrt{D_F}\,)^\varepsilon, $$ 

are fulfilled unconditionally for the family ${\mathscr F}_{\!\!\kappa}^{p^e}$ 
of $p$-cyclic-towers $F/\Kappa$ of degree~$p^e$. 
\end{theorem}

\proof
For any number field $L$, we have ${\rm rk}_p({\mathcal A}_L) =
r_2(L)+1 + {\rm rk}_p({\mathcal T}_L)$ and $r_2(L)+1$ is the free rank,
${\rm rk}_{\Z_p}^{}({\mathcal A}_L)$, of ${\mathcal A}_L$. Thus  we get, 
from relation (\ref{r22}) and ${\rm rk}_{\Z_p}^{}({\mathcal A}_K) = 
{\rm rk}_{\Z_p}^{}({\mathcal A}_k) + {\rm rk}_{\Z_p}^{}({\mathcal A}_K^*)$:
$$r_2(K)+1 + {\rm rk}_p({\mathcal T}_K) \leq
p\,(r_2(k)+1 + {\rm rk}_p({\mathcal T}_k)) + (p-1)\,N\,d_k, $$
whence, since $r_2(K) -  r_2(k) = {\rm rk}_{\Z_p}^{}({\mathcal A}_K^*) \geq 0$:

\medskip
\centerline{${\rm rk}_p({\mathcal T}_K) \leq p\,{\rm rk}_p({\mathcal T}_k) +
p\,r_2(k) + p - r_2(K) - 1  + (p-1)\,N\,d_k$}

\medskip
\centerline{\hspace{1.0cm}$\leq p\,{\rm rk}_p({\mathcal T}_k) + (p-1)\,r_2(k) + p-1 + (p-1)\,N\,d_k$}

\medskip
\centerline{\hspace{-2.0cm}$\leq p\,{\rm rk}_p({\mathcal T}_k)+(p-1)\,(N+1)\,d_k$}
 
\medskip
The inequalities are similar to that obtained for the $p$-ranks 
of usual class groups, whence the result since
${\rm rk}_p ({\rm H}^2({\mathcal G}_F,\Z_p)) = {\rm rk}_p ({\mathcal T}_F)$.
\qed \medskip

As for the $p$-class groups, we have, for $r \geq 1$ fixed:
$$\order ({\rm H}^2({\mathcal G}_F,\Z_p) \otimes \Z/p^r \Z)
\ll_{\kappa ,p^e,\varepsilon} (\sqrt{D_F}\,)^\varepsilon, \ \ 
\hbox{for all $F \in {\mathscr F}_{\!\!\kappa}^{p^e}$}. $$

\begin{remark} \label{divers}
{\rm Most arithmetic $p$-invariants, stemming from generalized class groups
$\Cl_\Sigma^S$ (regarding ramification and decomposition of given sets of places
in the corresponding ray class fields), fulfill the $p$-rank $\varepsilon$-conjecture 
for ${\mathscr F}_{\!\!\kappa}^{p^e}$.

In the same way, many other $p$-invariants are related to the fundamental
groups $\Cl_F$ and ${\mathcal T}_F$ by means of standard rank formulas
and/or dualities (e.g., reflection theorems detailed in \cite[Chapitre III]{Gr5}), 
so that all these invariants fulfill the $p$-rank 
$\varepsilon$-conjecture. One may cite:

\smallskip
(i) The normalized $p$-adic regulator ${\mathcal R}_F$ (obvious from the schema).

\smallskip
(ii) The regular kernel $R_F$ and the Hilbert kernel $W_F$ (from results of Tate; 
see \cite{Gr6} \cite[\S\,7.7.2, Theorem 7.7.3.1]{Gr0}, \cite{GJ});
then the corresponding study for the higher ${\rm K}$-theory \cite[\S\,12]{Gr5}.

\smallskip
(iii) The Jaulent logarithmic  class group \cite{BJ,J1,J2,J3} (finite under the conjecture of 
Gross and isomorphic to a quotient of ${\mathcal T}_F$), linked with precise formulas 
to the previous groups $\Cl_F$, ${\mathcal T}_F$, $W_F$ \cite{J1,J2}.}
\end{remark}

\section{Conclusion} 

The main question remains the case of a {\it strong $\varepsilon$-conjecture}, for 
such finite invariants $M$, saying that:

\medskip
\leftline{$\order (M_F \otimes \Z_p) \ll_{\kappa ,p^e,\varepsilon} (\sqrt{D_F}\,)^\varepsilon$
for all $F \in {\mathscr F}_{\!\!\kappa}^{p^e}$,}

\smallskip
\rightline{{\it except for a subfamily of ${\mathscr F}_{\!\!\kappa}^{p^e}$ of zero density}.}

\medskip
This restriction seems essential because of the existence of very rare 
fields giving exceptional large invariants $M$ as shown in \cite{Da,DaKM,Lam} for 
class groups (or \cite{Gr4} for torsion groups ${\mathcal T}$).
This is also justified, in the framework of $p$-class groups, by the Koymans--Pagano
density results \cite{KP} as analyzed in \cite{Gr00} for ${\mathscr F}_{\!\!\Q}^{p}$;
indeed, in any relative degree $p$ cyclic extension, the algorithm defining the
filtration $(M_h)_{h \geq 0}$ is a priori unbounded, giving possibly large $\order M$
contrary to the $p$-ranks (or the $\order M[p^r]$ as seen in Corollary \ref{Rank} 
which allows to take $r \gg 0$, but {\it constant} regarding the familly 
${\mathscr F}_{\!\!\kappa}^{p^e}$).

\smallskip
All the previous results on $p$-rank $\varepsilon$-inequalities
fall within the framework of ``genus theory'' at the prime $p$ 
for $p$-extensions; the case of degree $d$ number fields, when 
$p \nmid d$, is highly non-trivial.
For instance, the simplest case of the $3$-rank $\varepsilon$-conjecture
for quadratic fields $F$ remains open since one only knows that
$\order (\Cl_F \otimes \F_3) \ll_{\varepsilon} 
(\sqrt{D_F}\,)^{\frac{2}{3}+\varepsilon}$ (we refer to the bibliographies
of \cite{EV,W}, among other, for many generalizations and improvements 
of the exponents).

\smallskip
Indeed, the general case, regarding the $\ell$-invariants $M \otimes \F_\ell$,
$M \otimes \Z_\ell$ of $M$, in degree $d$ extensions, has a 
fundamental difficulty since complex analytic methods consider 
globally $\order M$ as upper bound (like ``$\order (M \otimes \F_\ell) 
\leq \order M$'', in the framwork of Brauer--Siegel type results 
\cite{N,PTW2,TV,Zy}) and often assume GRH, so that the ``bad primes'' 
$p\!\mid\! d$ may give large $p$-parts in $\order M$, thus analytic difficulties
as explained in \cite[\S\,2]{Gr00}; this is due to the lack of direct $p$-adic analytic 
tools.

\end{document}